\documentclass[10pt]{amsart}
\pdfoutput=1 

\usepackage{amsmath,amsthm,amssymb}
\usepackage[mathscr]{eucal}
 \usepackage{cite}
\usepackage{upgreek}
\usepackage[bookmarks=false]{hyperref}
\usepackage{enumerate}
\usepackage{mathrsfs}
\usepackage{blindtext}
\usepackage{scrextend}
\usepackage{bm} 
\addtokomafont{labelinglabel}{\sffamily}
\usepackage{color}
\usepackage[at]{easylist}
\usepackage{tikz}
\usepackage{svg}
\usepackage{graphicx}
\usepackage{mathtools}

\usepackage{comment}

\numberwithin{equation}{section}

\newtheorem{main}{Theorem}

\newtheorem{mcor}[main]{Corollary}

\newtheorem{thm}{Theorem}[section]
\newtheorem*{thm*}{Theorem}
\newtheorem{lem}[thm]{Lemma}
\newtheorem*{prob*}{Problem}

\newtheorem{prop}[thm]{Proposition}
\newtheorem*{prop*}{Proposition}

\newtheorem*{cor*}{Corollary}

\theoremstyle{definition}
\newtheorem{defn}[thm]{Definition}
\newtheorem*{defn*}{Definition}
\newtheorem{notation}[thm]{Notation}
\newtheorem{remark}[thm]{Remark}

\newtheorem*{question*}{Question}
\newtheorem*{Pquestion*}{Popa's question}

\newtheorem*{conv*}{Convention}

\newcommand{\N}{\mathbb{N}}

\newcommand{\C}{\mathbb{C}}
\newcommand{\Z}{\mathbb{Z}}

\newcommand{\F}{\mathbb{F}}

\newcommand{\cH}{\mathcal{H}}
\newcommand{\cI}{\mathcal{I}}

\newcommand{\cK}{\mathcal{K}}
\newcommand{\cM}{\mathcal{M}}

\newcommand{\cU}{\mathcal{U}}

\newcommand{\id}{\operatorname{id}}

\newcommand{\ootimes}{\mathbin{\overline{\otimes}}}

\DeclarePairedDelimiter{\ip}{\langle}{\rangle}

\begin{document}

\title[Exotic full factors via weakly coarse bimodules]
{Exotic full factors via weakly coarse bimodules}

\author[D. Gao]{David Gao}
\address{Department of Mathematical Sciences, UCSD, 9500 Gilman Dr, La Jolla, CA 92092, USA}\email{weg002@ucsd.edu}\urladdr{https://sites.google.com/ucsd.edu/david-gao}

\author[D. Jekel]{David Jekel}
\address{Department of Mathematical Sciences, University of Copenhagen, Universitetsparken 5, 2100 Copenhagen {\O}, Denmark}
\email{daj@math.ku.dk}
\urladdr{https://davidjekel.com}

\author[S. Kunnawalkam Elayavalli]{Srivatsav Kunnawalkam Elayavalli}
\address{Department of Mathematics, UMD, Kirwan Hall, Campus Drive, MD 20770, USA}\email{sriva@umd.edu}
\urladdr{https://sites.google.com/view/srivatsavke/home}

\author[G. Patchell]{Gregory Patchell}
\address{\parbox{\linewidth}{Mathematical Institute, University of Oxford, Andrew Wiles Building, \\ Radcliffe Observatory Quarter, Woodstock Road, Oxford, OX2 6GG, UK}}
\email{greg.patchell@maths.ox.ac.uk}
\urladdr{https://sites.google.com/view/gpatchel}

\begin{abstract}
        We are able to explicitly compute the bimodule structure of von Neumann algebra inclusions in \emph{handle constructions}, which arise as inductive limits of iterated amalgamated free products not elementarily equivalent to $L(\mathbb{F}_2)$. Our computation is achieved via identifying delicate normal form decompositions in amalgamated free products built in an iterated fashion. Using these techniques, we are able to show that the handles constructions are always full, without any need to appeal to Property (T) phenomena which was essential in all previous works. Furthermore our bimodule machinery works in the setting of arbitrary von Neumann algebras equipped with faithful normal states, yielding examples of full I$\mathrm{II}_1$ factors via handle constructions.
        
\end{abstract}

\maketitle

\section{Introduction}

The study of ultrapowers of von Neumann algebras goes back to the very early days of the field, and was a key ingredient in many landmark results \cite{DiLa, ZM, McD1, McD2, Connes}.  This naturally gave rise to the problem of determining when two von Neumann algebras have isomorphic ultrapowers, which is also known as \emph{elementary equivalence} (see \cite{FHS}).  While determining whether two von Neumann algebras are isomorphic is already a notoriously difficult problem, the same tools often cannot be used to assess the much coarser relation of elementary equivalence.  Distinguishing von Neumann algebras up to elementary equivalence requires new more robust invariants, and several such involving the central sequence algebra have been studied in \cite{BCI15, GH, goldbring2023uniformly}.  Much less is known about elementary equivalence for $\mathrm{II}_1$ factors that are \emph{full} or \emph{non-Gamma}, or do not admit nontrivial central sequences.  The first construction of full $\mathrm{II}_1$ factors that have been proven non-elementarily equivalent was given in \cite{CIKE23}.  The distinguishing invariant in this case is the \emph{sequential commutation diameter}; the framework of sequential commutation was developed in \cite{elayavalli2023sequential} (see also \cite{gao2024internalsequentialcommutationsingle}).

Fix a $\mathrm{II}_1$ factor $M$ and let $M^{\mathcal{U}}$ be its ultrapower.  Consider the graph whose vertices are the Haar unitaries in $M^{\mathcal{U}}$ (i.e.\ unitaries whose spectral distribution is uniform on the circle), where $u$ and $v$ are adjacent if $uv = vu$.  The connected components of this graph are \emph{sequential commutation orbits} and the supremum of the distances between two vertices $u$ and $v$ is the \emph{sequential commutation diameter} (this is also $+\infty$ by convention if the graph is disconnected).  For example, the sequential commutation diameter is $2$ if and only if $M$ has property Gamma (see \cite[Corollary 4.11]{shen2019reducible} and \cite[Theorem 4.2]{elayavalli2023sequential}).  A completely opposite case is the free group von Neumann algebra $L(\mathbb{F}_n)$, for which the commutation graph is disconnected, since the $n$ different generators of $\mathbb{F}_n$ are in different sequential commutation orbits \cite{Hayes8, jekel2024upgradedfreeindependencephenomena} (see also \cite{houdayer2023asymptotic} which proves the non-existence of commutation paths of length 3 on the generators).

The strategy of \cite{CIKE23} was to construct a $\mathrm{II}_1$ factor with commutation diameter exactly $3$; equivalently, any two unitaries can be connected by a path of length $3$ in the commutation graph and $M$ is full.  An upper bound on the commutation diameter can be arranged by an inductive construction; after enumerating a dense subset of pairs of unitaries $(u_n,v_n)$ in the starting algebra $M$, one can take an amalgamated free product over $\mathrm{W}^*(u_n,v_n)$ with another algebra that contains path of sequentially commuting Haar unitaries that connects $u_n$ to $v_n$.  One also iterates over a dense family of the unitaries in the larger algebras using a diagonalization argument, so the inductive limit has a desired upper bound on the commutation diameter.

Perhaps the most challenging aspect of this construction is how to ensure that the inductive limit is full, or that we have not added ``too many'' commuting unitaries.  The first approach in \cite{CIKE23} was very technical and involved iterating only over pairs of unitaries $(u,v)$ with $u^2 = v^3 = 1$ such that the algebras they generate are orthogonal.  This required lifting theorems for such unitaries as well as an involved deformation--rigidity manipulation using \cite{IPP}. The work \cite{houdayer2023asymptotic} was then able to remove the condition on the pairs of independent unitaries with a more refined lifting theorem. The second approach in \cite{gao20253handleconstructionii1factors} added a sequentially commuting path of length four rather than length three at each stage of the induction.  This ``$3$-handle construction'' resulted in a substantially easier argument to obtain fullness, only requiring the initial input algebra to have property (T) and straightforwardly applying \cite{IPP}, with no need for technical lifting theorems \`a la \cite{CIKE23,houdayer2023asymptotic}.

While the second approach had a much shorter proof, the requirement of property (T) on the base algebra still remained a significant limitation on what kinds of full $\mathrm{II}_1$ factors can be produced.  It also prevents us from generalizing the construction to the type $\mathrm{III}$ setting where property (T) in the na{\"\i}ve sense can never hold.  However, property (T) naturally arises in the course of the argument.   Indeed, by \cite{tan2023spectral}, a $\mathrm{II}_1$ factor $M$ having property (T) is equivalent to every irreducible embedding $M \subseteq N$ having weak spectral gap (that is, $M' \cap N = \C$ implies $M' \cap N^{\mathcal{U}} = \C$).  In \cite{CIKE23} and \cite{gao20253handleconstructionii1factors}, this is what guarantees that the larger algebra obtained from the inductive construction is still full.  Thus, without property (T), we cannot transfer fullness to the larger algebra without using more particular properties of the embedding.  In light of this as well as Peterson's conjecture \cite[Problem U5]{JessePL}, it was essentially unclear whether property (T) was playing a conceptual role in obtaining non-elementarily equivalent full $\mathrm{II}_1$ factors.

This paper will give a new argument for fullness that avoids property (T) entirely.  Instead, we use the well-known fact that if $M \subseteq N$ is an inclusion of factors and $L^2(N) \ominus L^2(M)$ is a weakly coarse bimodule over $M$, then fullness of $M$ implies fullness of $N$ (see for instance \cite{ioana2023existential}).  We thus obtain the following main result.

\begin{main} \label{main: fullness}
Let $M$ be any $\mathrm{II}_1$ factor.  Let $\Theta(M)$ be obtained from $M$ by the iterated $3$-handle construction of \cite{gao20253handleconstructionii1factors} (see \S \ref{sec: exotic factor construction} below).  Then $\Theta(M)$ is a full $\mathrm{II}_1$ factor.
\end{main}

Aside from opening up new possibilities for elementary equivalence results, the above result yields the following quick corollary which demonstrates the flexibility of being able to choose an arbitrary input factor.

\begin{mcor} \label{main: no relative T}
$\Theta(L(\mathbb{F}_2))$ is a full $\mathrm{II}_1$ factor that is not elementarily equivalent to $L(\mathbb{F}_2)$, and does not contain any diffuse property (T) subalgebra.
\end{mcor}

Our approach is based on a new explicit calculation of the \emph{bimodule structure} of the iterated amalgamated free products occurring in the $3$-handle construction of \cite{gao20253handleconstructionii1factors}.  One step in the $3$-handle construction is performed as follows.  We take as input a $\mathrm{II}_1$ factor $N_0$ and two unitaries $u_1,u_2\in \mathcal{U}(N_0)$.  Then set
\begin{align*}
N_1 &= N_0*_{\{u_1\}''}(\{u_1\}''\ootimes L(\mathbb{Z})) \\
N_2 &= N_1 *_{\{u_2\}''}({\{u_2\}''}\ootimes L(\mathbb{Z})) \\
N &=N_2*_{\{v_1, v_2\}''}(\{v_1, v_2\}''\ootimes L(\mathbb{Z})),
\end{align*}
where $v_1$, $v_2$ are the generators of the two copies of $L(\mathbb{Z})$ in $N_2$.  We are able to prove that $N_0\subset N$ is \emph{weakly coarse}; i.e, the bimodule $_{N_0}[L^2(N)\ominus L^2(N_0)]_{N_0}$ is \emph{weakly contained} in the coarse $N_0$-bimodule, namely $_{N_0}[L^2(N_0) \otimes L^2(N_0)^{\text{op}}]_{N_0}$.  The novel part of the argument is the coarseness of $N_2$ inside $N$, which uses subtle word decomposition arguments.  We conclude that fullness of $N_0$ implies fullness of $N$, and the same holds for the inductive limit since weak coarseness passes to inductive limits, allowing us to show Theorem A.

To prove Corollary \ref{main: no relative T} from this point, we begin the handle construction with $L(\mathbb{F}_2)$, use standard techniques in deformation rigidity and perturbation to precisely locate $(T)$ subalgebras, and use the Haagerup property of $L(\mathbb{F}_2)$ and \cite{ConnesJones} to conclude.  

Our bimodule computation also does not rely on the state being tracial, so it allows us to prove fullness for such constructions on von Neumann algebras with faithful normal states.  In particular, we prove fullness of type $\mathrm{III}_1$ factors obtained from an iterated $3$-handle construction, even though there is no satisfactory of analog of property (T) in the type $\mathrm{III}$ setting.  However, there is an added subtlety in the type $\mathrm{III}$ setting that we must restrict to pairs of unitaries $(u_1,u_2)$ such that $\{u_j\}'' \subseteq M$ admits a state-preserving conditional expectation (or at least $u_j$ is contained in some amenable subalgebra with state-preserving conditional expectation).  This restriction on the unitaries prevents us from immediately concluding something about elementary equivalence in the type $\mathrm{III}$ setting, although it seems reasonable to conjecture that the algebras arising from the iterated $3$-handle construction would be non-elementarily equivalent to free Araki--Woods factors (using \cite{hayes2024generalsolidityphenomenaanticoarse}).

\subsection*{Acknowledgements}

This work began in 2023 at UCSD. The first, third and fourth authors thank the Torrey Pines Gliderport wherein several meetings took place. The project was completed during a visit of the second author to the University of Maryland supported by the Brin Mathematics Research Center. We are very grateful to Adrian Ioana and Jesse Peterson for their constant encouragement and insightful suggestions. We are also grateful to Soham Chakraborty, Amine Marrakchi, Brent Nelson for some helpful correspondences. We thank Jesse Peterson again for pointing out a mistake in an earlier draft.

\subsection*{Funding}

The second author is supported by the Horizon Europe Marie Sk{\l}odowska Curie Action FREEINFOGEOM, project 101209517.\footnote{Views and opinions expressed are those of the author(s) only and do not necessarily reflect those of the European Union or the Research Executive Agency. Neither the European Union nor the granting authority can be held responsible for them.}
The third author is supported by NSF grant DMS 2350049. The fourth author was supported by the Engineering and Physical Sciences Research Council (UK), grant EP/X026647/1.  

\subsection*{Open Access and Data Statement}

For the purpose of Open Access, the authors have applied a CC BY public copyright license to any Author Accepted Manuscript (AAM) version arising from this submission. Data sharing is not applicable to this article as no new data were created or analyzed in this work.

\section{Preliminaries} \label{sec: preliminaries}

To keep the article concise, we will assume the reader is familiar with the standard theory of von Neumann algebras (see for instance \cite{anantharaman-popa, BrownOzawa2008}). 

\subsection{Ultrapowers of general von Neumann algebras}

\begin{defn}
    Let $M$ be a von Neumann algebra with separable predual and with a normal faithful state $\varphi$ and let $\cU$ be a free ultrafilter on $\N$. We make the following definitions as in Sections 2 and 3 of \cite{AndoHaag2014typeIIIultra} (see also \cite[Section~5]{Oceanu1985}):
    \begin{itemize}
        
        \item $\ell^\infty(\N,M) := \{(x_n)_n \in \prod_{n\in\N}M : \sup \|x_n\| < \infty\}$;
        \item $\|x\|_\varphi^\sharp := \varphi(xx^* + x^*x)^{1/2}$;
        \item $\cI_\cU := \{(x_n)_n\in \ell^\infty(\N,M) : \lim_{n\in\cU}\|x_n\|_\varphi^\sharp = 0\}$;
        \item $\cM^\cU := \{(x_n)_n\in \ell^\infty(\N,M) : (x_n)_n\cI_\cU\subset\cI_\cU \text{ and }\cI_\cU(x_n)_n\subset\cI_\cU \}$;
        \item $M^\cU := \cM^\cU /\cI_\cU$;
        \item $\cM_\cU := \{(x_n)_n\in \ell^\infty(\N,M) : \|x_n\psi -\psi x_n\| \to 0 \text{ for all } \psi\in M_* \}$;
        \item $M_\cU := \cM_\cU / \cI_\cU$.
    \end{itemize}
\end{defn}

The ideal $\cI_\cU$ in fact does not depend on the choice of $\varphi,$ so we are justified in the notations $M^\cU$ and $M_\cU$. In general, we have the inclusion $M_\cU \subset M'\cap M^\cU$. In the case $M_\cU = \C,$ we also have $M'\cap M^\cU = \C$ by \cite[Theorem~5.2]{AndoHaag2014typeIIIultra}. This theorem justifies the following definition of fullness.

\begin{defn}
A von Neumann algebra $M$ is said to be \emph{full} if $M'\cap M^\cU = \C1$.
\end{defn}

\subsection{Bimodules over general von Neumann algebras}

Implicitly, we endow all von Neumann algebras herein with a faithful, normal state. We say that a von Neumann subalgebra $N\subset M$ is \emph{with expectation} if there is a normal state-preserving conditional expectation $E: M\to N$.  Note a state-preserving expectation must be faithful if the state is faithful. 

We recall the following definitions facts about bimodules over von Neumann algebras (see \cite[Appendix F]{BrownOzawa2008}, \cite{OOTHaag}, \cite{BMOCoAmenable}):
\begin{itemize}
    \item For von Neumann algebras $M$ and $N$, an \emph{$M$-$N$-bimodule} is a Hilbert space $H$ equipped with a normal left action of $M$ and right action of $N$.
    \item For a von Neumann algebra $M$ with faithful normal state $\varphi$, the \emph{trivial bimodule} $L^2(M,\varphi)$ (often denoted simply $L^2(M)$) is the GNS space equipped with the left and right actions $x \cdot \xi \cdot y = x JyJ \xi$ where $J$ is the modular conjugation operator.
    \item \emph{The coarse $M$-$N$ bimodule} is the bimodule $L^2(M) \otimes_{\C} L^2(N)$.  More generally, an $M$-$N$-bimodule $H$ is said to be \emph{coarse} if $M$ embeds into a direct sum of copies of $L^2(M) \otimes_{\C} L^2(N)$.
    \item An $M$-$N$-bimodule $\cH$ is \emph{weakly contained} in another $M$-$N$-bimodule $\cK$ if $\cH$ is in the closure of $\cK^{\oplus\infty}$ in the Fell topology (see \cite[Section~13.3.2]{anantharaman-popa}).  In particular, an $M$-$N$ bimodule is \emph{weakly coarse} if it is weakly contained in the coarse bimodule.
    \item Given bimodules $_{M_1} H_{M_2}$ and $_{M_2} K_{M_3}$, one may form a \emph{relative tensor product} or \emph{Connes fusion} $_{M_1} H \otimes_{M_2} K_{M_3}$; see \cite{OOTHaag}.  In the special case that $M_2 \subseteq M_1$ with expectation $E_{M_2}$ and $H = {}_{M_1} L^2(M_1)_{M_2}$, the bimodule ${}_{M_1} L^2(M_1) \otimes_{M_2} K_{M_3}$ is described as the completion of $M_1 \otimes_{\operatorname{alg}} K$ with respect to the inner product given by
    \[
    \ip{x \otimes \xi, y \otimes \eta} = \ip{\xi, E_{M_2}[x^*y] \eta}_K.
    \]
\end{itemize}

The following two facts are well known to experts, but we include a proof for the benefit of the reader. 

\begin{prop}\label{prop: full-and-coarse}
Let $N$ be a full factor and $N\subset M$ be a von Neumann subalgebra with expectation such that $_N[L^2(M)\ominus L^2(N)]_N$ is weakly contained in the coarse $N$-$N$-bimodule. Then $M$ is a full factor.
\end{prop}

\begin{proof}
    Let $(x_n)_n \in M' \cap M^\cU$. Let $E:M\to N$ be the conditional expectation. Set $y_n = x_n - E(x_n)$. Clearly, $y_n \in L^2(M)\ominus L^2(N)$. Moreover, it is easy to check that since $(x_n)_n$ is asymptotically $M$ (and therefore $N$)-central, it follows that $(y_n)_n$ is asymptotically $N$-central too. However, weak coarseness of $_N[L^2(M)\ominus L^2(N)]_N$ implies that it cannot contain any asymptotically $N$-central vectors. Therefore $y_n \to 0$ in 2-norm. Hence $x_n \to E(x_n)$ in 2-norm. But $E(x_n) \in N$ is asymptotically $N$-central, so by fullness of $N,$ $E(x_n)_n$ (and thus $(x_n)_n$) are asymptotically trivial. That is, $M$ is also full.
\end{proof}

\begin{prop}\label{prop: lim-of-coarse}
    Let $N_i\subset N_{i+1}$ be a chain of inclusions of von Neumann algebras for $i\in \mathbb{N}$, and further $_{N_{i}}[L^2(N_{i+1})\ominus L^2(N_i)]_{N_{i}}$ is weakly coarse. Then $N_0\subset \overline{\bigcup_{i\in \mathbb{N}}N_i}^{SOT}$ is weakly coarse.
\end{prop}

\begin{proof}
    We first show by induction that $L^2(N_k)\ominus L^2(N_0)$ is weakly coarse for all $k.$ The base case $k=1$ is by hypothesis. Otherwise, suppose $L^2(N_k)\ominus L^2(N_0)$ is weakly coarse. Note that
    $$L^2(N_{k+1})\ominus L^2(N_0) = (L^2(N_{k+1})\ominus L^2(N_k)) \oplus (L^2(N_k)\ominus L^2(N_0)).$$
    Therefore, it suffices to show that $L^2(N_{k+1})\ominus L^2(N_k)$ is weakly coarse as an $N_0$-$N_0$-bimodule. By hypothesis, it is weakly coarse as an $N_k$-$N_k$-bimodule, so it suffices to show that $L^2(N_k)\otimes L^2(N_k)$ is (weakly) coarse as an $N_0$-$N_0$-bimodule. But this is clear since as a left (and right) $N_0$-module $L^2(N_k)$ is a direct sum of copies of $L^2(N_0)$.

    Now denote by $N$ the inductive limit $\overline{\bigcup_{i\in \mathbb{N}}N_i}^{SOT}$. Note that $L^2(N) \ominus L^2(N_0)$ is the closed span of $L^2(N_k)\ominus L^2(N_0)$ as $k$ ranges over $\N.$ Since weak containment is a property checked only by finitely many vectors and with an $\varepsilon$ tolerance, it is immediate that $L^2(N)\ominus L^2(N_0)$ is weakly coarse too.
\end{proof}

\subsection{Amalgamated free products}

Let $B\subset M_1,M_2$ be a subalgebra with expectation. Then we may form the \emph{amalgamated free product} $M_1*_BM_2$. We follow Section 3 of \cite{Popa93} (see also \cite[Section~2]{GKEPTconjugacy2025}). The von Neumann algebra $M_1 *_B M_2$ is SOT-densely spanned by elements of the form $x=b\in B$ and $x = x_{1,i_1}\cdots x_{k,i_k}$ where $x_{j,i_j}\in M_{i_j}$, $E_{i_j}(x_{j,i_j}) = 0$ (in other words, $x_{j,i_j}\in M_i\ominus B$), and $i_1\neq \ldots \neq i_k$. We call such elements reduced words. We note that the subspaces $B$ and $(M_{i_1}\ominus B)\cdots(M_{i_k}\ominus B)$ are all orthogonal for different tuples $(i_1,\ldots,i_k)$ such that $i_1\neq \ldots \neq i_k$. Multiplication on reduced words is defined by concatenation, and the linear span of reduced words is closed under multiplication via the simplification operation $(x,y)\in M_i^2 \mapsto E_B(xy) + (xy- E_B(xy))$. The amalgamated free product $M_1*_BM_2$ is typically represented on the following Hilbert space, often also viewed as a $B$-$B$-bimodule.

\begin{align*}
    L^2(M_1*_BM_2) = L^2(B) \oplus \bigoplus_{n\ge1}\bigoplus_{i_1\neq \ldots\neq i_n} &(L^2(M_{i_1})\ominus L^2(B))\otimes_B \\
    &\cdots \otimes_B (L^2(M_{i_n})\ominus L^2(B)).
\end{align*}
The following then follows from a straighforward computation.

\begin{lem}\label{lem: afp-bimod}
    As an $M_1$-$M_1$-bimodule, we have 
    \begin{align*}
        L^2(M_1*_BM_2) &= L^2(M_1) \oplus \bigoplus_{n\ge0} L^2(M_1) \otimes_B (L^2(M_2)\ominus L^2(B)) \\ &\otimes_B ((L^2(M_1)\ominus L^2(B)) \otimes_B (L^2(M_2)\ominus L^2(B)))^{\otimes_Bn} \otimes_B L^2(M_1).
    \end{align*}
\end{lem}

This implies the following well-known coarseness result for amalgamated free products over amenable subalgebras.

\begin{prop}\label{prop: amen-afp-wk-coarse}
    Let $N_1,N_2$ be von Neumann algebras and $A$ an amenable subalgebra of both $N_1,N_2$ with expectation. Set $M=N_1*_A N_2$. Then $N_1\subset M$ is weakly coarse.
\end{prop}

\begin{proof}
    Since $A$ is amenable, we have that $L^2(A)\prec L^2(A)\otimes L^2(A)$. Therefore, for any $N_1$-$N_1$-bimodules $\cH,\cK$, we have that 
    $$\cH\otimes_A\cK = \cH\otimes_A L^2(A) \otimes_A\cK \prec \cH\otimes_A L^2(A)\otimes L^2(A)\otimes_A\cK = \cH\otimes\cK.$$
    Set $\cH_1 = L^2(N_1) \ominus L^2(A)$ and $\cH_2 = L^2(N_2)\ominus L^2(A)$. Now, by Lemma \ref{lem: afp-bimod}, we have 
    \begin{align*}
L^2(M) \ominus L^2(N_1)
\,\, \cong \,\, & \bigoplus_{k \geq 0} L^2(N) \otimes_A \cH_2 \otimes_{A} (\cH_1 \otimes_A \cH_2)^{\otimes_A k} \otimes_A L^2(N)\\
\prec \, \, & \bigoplus_{k \geq 0} L^2(N) \otimes \cH_2 \otimes (\cH_1 \otimes \cH_2)^{\otimes k} \otimes L^2(N)\\
\prec \, \, & L^2(N) \otimes L^2(N).
\end{align*}
\end{proof}

\section{Bimodule computations} \label{sec: bimodule computation}

We consider the following setup which generalizes the construction of $N_0 \subseteq N_1 \subseteq N_2$ in the introduction.  This would correspond to the case of Notation \ref{not: free product setup} below where $n = 2$ and $B_j = \{u_j\}''$ and $C_j = \{v_j\}''$.  The addition of the last unitary for the $3$-handle construction that commutes with the $v_j$'s will be handled at the last step in Theorem \ref{thm: coarseness 2}.

\begin{notation} \label{not: free product setup}
	Let $n \geq 1$.  For $j=1, \dots, n$, let $(A,\varphi)$ be a von Neumann algebra with faithful normal state and $B_j \subset B$ a von Neumann subalgebra with state-preserving conditional expectation $E_{B_j}$; let $\varphi_j = \varphi|_{B_j}$.  Let $(C_j,\psi_j)$ be another von Neumann algebra with faithful normal state.  Define von Neumann algebras $M_0 \subset M_1 \subset \dots \subset M_n$ with states $\omega_j$ inductively by
	\begin{align*}
		(M_0,\omega_0) &= (A,\varphi) \\
		(M_{j+1},\omega_{j+1}) &= (M_j,\omega_j) *_{(B_j,\varphi_j)} [(B_j,\varphi_j) \ootimes (C_j,\psi_j)],
	\end{align*}
	where $(B_j,\varphi_j) \subset (A,\varphi) = (M_0,\omega_0)$ is regarded as a subalgebra of $(M_j,\omega_j)$ in the natural way, and we also use the natural inclusion of $(B_j,\varphi_j)$ into $(B_j,\varphi_j) \ootimes (C_j,\psi_j)$ with the conditional expectation given by $\id \otimes \psi_j$.  We also regard $(A,\varphi)$, $(B_j,\varphi_j)$, and $(C_j,\psi_j)$ as subalgebras of $(M_n,\omega_n)$ in the natural way.
\end{notation}

\begin{thm} \label{thm: coarseness 1}
	Consider the setup of Notation \ref{not: free product setup}.  Then $C_1$, \dots, $C_n$ are freely independent in $(M_n,\omega_n)$.  This induces an embedding of $(C,\psi) = (C_1,\psi_1) * \dots * (C_n,\psi_n)$ into $(M_n,\omega_n)$, which is with expectation.  Moreover, there exists an isomorphism of $C$-$A$-bimodules
	\[
	{}_C L^2(M_n,\omega_n)_A \to {}_C L^2(C,\psi) \otimes_{\mathbb{C}} H_{A}
	\]
	for some right $A$-module $H$.
\end{thm}

\begin{remark}
In special cases, this theorem overlaps with the bimodule computations for graph products in \cite[\S 5]{charlesworth2023strong2}.  Namely, if $A$ is the graph product of the algebras $B_j$ over some graph on vertex set $[n]$, then $(M_n,\omega_n)$ would be the graph product of the $B_j$'s and the $C_j$'s with respect to a graph on vertex set $[2n]$ obtained by adding for each $j \in [n]$ an extra vertex $n+j$ adjacent to $j$.  Then $A$ and $C$ would be subalgebras induced from disjoint subgraphs, and so coarseness would follow from  \cite[Theorem 5.4]{charlesworth2023strong2}.  The computation here is thus inspired by the graph product case, but significantly more general since in the present setting there is not a canonical or unique word decomposition of elements of $M_n$.
\end{remark}

\begin{lem} \label{lem: word decomposition}
	Consider the setup of Notation \ref{not: free product setup}.  For each alternating word $w = j_1 \dots j_\ell$ on the alphabet $\{0,\dots,n\}$, let $K_w$ be the span of products of the form $x_1 \dots x_\ell$, where
	\begin{enumerate}[(a)]
		\item if $j_i > 0$, then $x_i \in C_{j_i}$ with $\psi_{j_i}(x_i) = 0$;
		\item if $j_i = 0$, then $x_i \in A$;
		\item if $j_i = 0$ and $i < \ell$, then $E_{B_{j_{i+1}}}[x_i] = 0$.
	\end{enumerate}
	Then the $K_w$'s span a weakly dense subset of $M_n$.
\end{lem}

\begin{proof}
	An easy induction shows that $M_j$ is generated by $A$ and $B_{j'} \ootimes C_{j'}$ for $j' \leq j$.  Since $B_j \subset A$, $M_n$ is clearly also generated by $A$, $C_1$, \dots, $C_n$.  Hence, a weakly dense subset of $M_n$ is spanned by products $x_1 \dots x_\ell$ where $x_i$ is either in $A$ or one of the $C_j$'s.
	
	We claim that the span of products of length $\leq \ell$ of elements of $A$ and $C_1$, \dots, $C_n$ is equal to the span of the $K_w$'s where $w$ is alternating word on $\{0,\dots,n\}$ of length $\leq \ell$.  We proceed by induction on $\ell$.  The base case $\ell = 1$ is immediate since an element of $A$ is already an element of $A$, and an element of $x \in C_j$ can be expressed as $(x - \psi_j(x)) + \psi_j(x)1$, where $x - \psi_j(x)$ is in $\ker(\psi_j)$ and $\psi_j(x)1 \in A$.
	
	For the induction step, consider a product $x_1 \dots x_\ell$ of elements from $A$ and $C_1$, \dots, $C_n$, with $\ell \geq 2$.  By induction hypothesis, $x_1 \dots x_{\ell - 1}$ is in the span of words of the $K_w$'s for alternating words $w$ of length $\leq \ell - 1$. For each term that has length $\leq \ell - 2$, multiplying it by $x_\ell$ on the right yields a product of length $\leq \ell - 1$, which is already handled by the induction hypothesis. Hence, we may assume without loss of generality that $x_1 \dots x_{\ell - 1}$ satisfies the defining conditions of $K_w$ for some $w$ of length $\ell - 1$.
	
	If $x_\ell$ comes from the same algebra as $x_{\ell - 1}$ (either $A$ or $C_1$, \dots, $C_n$), then we can view $x_{\ell - 1} x_\ell$ as a single element in the product, and so $x_1 \dots x_\ell$ is a word of length $\leq \ell - 1$, which is already handled by the induction hypothesis.  Therefore, assume that $x_\ell$ and $x_{\ell - 1}$ come from different algebras.  Thus, there is an alternating word $w = j_1 \dots j_\ell$ satisfying that $x_i \in C_{j_i}$ when $j_i > 0$ and $x_i \in A$ when $i = 0$. 
	
	\textbf{Case 1:} Suppose that $j_\ell = 0$ so that $x_\ell \in A$.  If $x_1 \dots x_{\ell-1}$ satisfies conditions (a)-(c) for the word $j_1 \dots j_{\ell-1}$, then $x_1 \dots x_\ell$ satisfies conditions (a)-(c) for $j_1 \dots j_\ell$.  Indeed, there is nothing more to check for condition (a) since $j_\ell = 0$.  There is nothing more to check for condition (b) since $j_\ell = 0$ and $x_\ell \in A$.  There is nothing more to check about $x_{\ell-1}$ for condition (c) since $i_{\ell-1} \neq 0$.  Hence, $x_1 \dots x_\ell \in K_w$.
	
	\textbf{Case 2:} Next, suppose that $j_\ell > 0$ and so $x_\ell \in C_{j_\ell}$.  We can write $x_\ell = \mathring{x_\ell} + \psi_{j_\ell}(x_\ell) 1$ where $\mathring{x_\ell} \in \ker(\psi_{j_\ell})$.  Then $x_1 \dots x_{\ell-1} \psi_{j_\ell}(x_\ell) $ is a product of length $\ell - 1$, hence already handled by the induction hypothesis.  Thus, assume without loss of generality that $x_\ell = \mathring{x_\ell} \in \ker(\psi_{j_\ell})$.  This means that $x_1 \dots x_\ell$ already satisfies conditions (a) and (b).  
	
	\textbf{Case 2a:} Suppose that $i_{\ell-1} > 0$.   Then there is nothing more to check for condition (c) and so $x_1 \dots x_\ell \in K_w$ as desired.
	
	\textbf{Case 2b:} Suppose that $i_{\ell-1} = 0$.  Then write $x_{\ell-1} = \mathring{x_{\ell-1}} + E_{B_{j_\ell}}[x_{\ell-1}]$, where $E_{B_{j_\ell}}[\mathring{x_{\ell-1}}] = 0$.  Then $x_1 \dots x_{\ell-2} \mathring{x_{\ell-1}} x_\ell$ satisfies condition (c) and hence is in $K_w$.  It remains to shows that $x_1 \dots x_{\ell-2} E_{B_{j_\ell}}[x_{\ell-1}] x_\ell$ is in the span of the $K_w$'s.  Note that $E_{B_{j_\ell}}[x_{\ell-1}]$ commutes with $x_\ell \in C_{j_\ell}$ since $B_{j_\ell}$ and $C_{j_\ell}$ are in tensor position. Hence,
	\[
	x_1 \dots x_{\ell-2} E_{B_{j_\ell}}[x_{\ell-1}] x_\ell = x_1 \dots x_{\ell-2} x_\ell E_{B_{j_\ell}}[x_{\ell-1}].
	\]
	If $\ell \geq 3$ and $j_{\ell-2} = j_\ell$, then $x_1 \dots x_{\ell-3}(x_{\ell-2} x_\ell) E_{B_{j_\ell}}[x_{\ell-1}]$ becomes a product of length $\ell - 1$ handled by the induction hypothesis, by viewing $x_{\ell-2} x_\ell$ as a single element. Otherwise, $x_1 \dots x_{\ell-2} x_\ell$ satisfies conditions (a)-(c), and hence so does $x_1 \dots x_{\ell-2} x_\ell E_{B_{j_\ell}}[x_{\ell-1}]$ by Case 1.
\end{proof}

The following is well known:
\begin{lem} \label{lem: amalgamated free product decomposition}
	Let $(B,\psi) \subset (A,\varphi)$ and $(B,\psi) \subset (C,\rho)$ be state-preserving inclusions with expectations.  Let
	\[
	(M,\omega) = (A,\varphi) *_{(B,\psi)} (C,\rho).
	\]
	Consider a product $a_0 c_1 \dots a_{\ell-1} c_\ell a_\ell$ where
	\begin{enumerate}[(a)]
		\item $\ell \geq 1$
		\item $a_j \in A$.
		\item $E_B[a_j] = 0$ for $0 < j < \ell$.
		\item $c_j \in C$ with $E_B[c_j] = 0$ for $j = 1$, \dots, $\ell$.
	\end{enumerate}
	Then $E_A[a_0 c_1 a_1 \dots c_\ell a_\ell] = 0$.
	
	Let $a_0' c_1' \dots a_{m-1}' c_m' a_m'$ be another product satisfying the analogous conditions.
	\begin{enumerate}[(1)]
		\item If $m = \ell$, then
		\begin{align*}
			&E_A[a_\ell^* c_{\ell-1}^* a_{\ell-1}^* \dots c_1^* a_0^* a_0' c_1' \dots a_{\ell-1}'c_{\ell-1}' a_\ell']\\
			= &a_\ell^* E_B[c_{\ell-1}^*E_B[a_{\ell-1}^* \dots E_B[c_1^* E_B[a_0^* a_0'] c_1'] \dots a_{\ell-1}'] c_{\ell-1}'] a_\ell'.
		\end{align*}
		\item If $m \neq \ell$, then
		\[
		E_A[a_\ell^* c_{\ell-1}^* a_{\ell-1}^* \dots c_1^* a_0^* a_0' c_1' \dots a_{m-1}'c_{m-1}' a_m'] = 0.
		\]
	\end{enumerate}
\end{lem}	

\begin{lem} \label{lem: word decomposition 2}
	Consider the setup of Notation \ref{not: free product setup}, and let $K_w$ be as in Lemma \ref{lem: word decomposition}.  Suppose $w = j_1 \dots j_\ell$ is an alternating word on $\{0,\dots,n\}$.  Consider a product $x_1 \dots x_\ell$ satisfying
	\begin{itemize}
		\item[\emph{(a)}] if $j_i > 0$, then $x_i \in C_{j_i}$ with $\psi_{j_i}(x_i) = 0$;
		\item[\emph{(b)}] if $j_i = 0$, then $x_i \in A$;
		\item[\emph{(c')}] if $1 < i < \ell$ with $j_i = 0$ and $j_{i+1} = j_{i-1}$, then $E_{B_{j_{i-1}}}[x_i] = 0$.
	\end{itemize}
	Then
	\begin{enumerate}[(1)]
		\item If the word $w$ contains the letter $n$ at least once, then $E_{M_{n-1}}[x_1 \dots x_\ell] = 0$.
		\item Moreover, if the word $w$ contains any nonzero letter, then $E_A[x_1 \dots x_\ell] = 0$.
	\end{enumerate}
\end{lem}

\begin{proof}
	(1) We proceed by induction on $n$.  The proof of the base case $n = 1$ and the proof of the induction step will be handled at the same time.
	
	Let $n \geq 1$ and consider $x_1 \dots x_\ell$ as above.  Let $i(1) < \dots < i(k)$ be the indices where the letter $n$ appears in $w$.  By hypothesis $E_{B_n}[x_{i(t)}] = \psi_n(x_{i(t)}) = 0$ for $t = 1$, \dots, $k$.
	
	We also claim that $E_{B_n}[x_{i(t)+1} \dots x_{i(t+1)-1}] = 0$ for $t = 1$, \dots, $k-1$.  Let $m(t) = \max(j_{i(t)+1}, \dots, j_{i(t+1)-1})$ be the maximum index that appears in the subword $j_{i(t)+1} \dots j_{i(t+1)-1}$.  Note $m(t) < n$.  If $m(t) > 0$, then by induction hypothesis, $E_{M_{m(t)-1}}[x_{i(t)+1} \dots x_{i(t+1)-1}] = 0$ and therefore $E_{B_n}[x_{i(t)+1} \dots x_{i(t+1)-1}] = 0$ since $B_n \subset A \subset M_{m(t)-1}$. On the other hand, suppose that $m(t) = 0$. Then the subword can only have one letter $0$, and so $i(t+1) = i(t) + 2$. Hence, our goal is to show that $E_{B_n}[x_{i(t)+1}] = 0$.  Since $j_{i(t)+1} = 0$ and  $j_{i(t)} = j_{i(t)+2}= n$, condition (c') implies that $E_{B_n}[x_{i(t)+1}] = 0$, as desired.
	
	For $t = 1$, \dots, $k$, let $c_t = x_{i(t)}$ in $C_n \subset B_n \ootimes C_n$.  Let $a_t \in M_{n-1}$ be given by
	\[
	a_t = \begin{cases}
		x_1 \dots x_{i(1)-1}, & t = 0 \\
		x_{i(t)+1} \dots x_{i(t+1)-1}, & 0 < t < k \\
		x_{i(k)+1} \dots x_\ell, & t = k;
	\end{cases}
	\]
	for $t = 0$, if $i(1) = 1$, then $x_1 \dots x_{i(1)-1}$ is an empty product interpreted as $1$ (and we also have the symmetric statement for $t = k$).
	Then Lemma \ref{lem: amalgamated free product decomposition} applies to show that
	\[
	E_{M_{n-1}}[x_1 \dots x_\ell] = E_{M_{n-1}}[a_0 c_1 a_1 \dots c_k a_k] = 0.
	\]
	
	(2) If $m$ is the largest letter that appears in the word $w$, then $E_{M_{m-1}}[x_1 \dots x_\ell] = 0$ by (1), and since $A \subset M_{m-1}$ with expectation, this implies that $E_A[x_1 \dots x_\ell] = 0$.
\end{proof}

\begin{remark} \label{rem: relabeling indices}
	Note that condition (c) in Lemma \ref{lem: word decomposition} implies condition (c') in Lemma \ref{lem: word decomposition 2}, respectively. Hence, these two lemmas together allow one to determine $\omega_n(x_1 \dots x_\ell)$ for any product of elements from $A$ and $C_1$, \dots, $C_n$.  Furthermore, even though we constructed $M_n$ in Notation \ref{not: free product setup} by an iterated amalgamated free product, we would obtain the same result if we relabeled the indices $1$, \dots, $n$ to perform the amalgamated free products in a different order, i.e., the state on a product of elements from $A$ and $C_1$, \dots, $C_n$ would be the same for both constructions. This is because the conditions (a)-(c) in Lemma \ref{lem: word decomposition} are invariant under relabeling the indices $1$, \dots, $n$.
\end{remark}

\begin{lem} \label{lem: freeness of C}
$C_1$, \dots, $C_n$ are freely independent in $(M_n,\omega_n)$.  This induces an embedding of $(C,\psi) = (C_1,\psi_1) * \dots * (C_n,\psi_n)$ into $(M_n,\omega_n)$, which is with expectation.
\end{lem}

\begin{proof}
If $w = j_1 \dots j_\ell$ is an alternating word over $\{1,\dots,n\}$, and if $x_i \in C_{j_i}$ with $\psi_{j_i}(x_i) = 0$, then $x_1 \dots x_\ell$ satisfies the hypotheses of Lemma \ref{lem: word decomposition 2}; indeed, there is nothing to check for conditions (b) and (c) since the letter $0$ does not appear.  Therefore, $\omega_n(x_1 \dots x_\ell) = 0$.  Standard results show that there is a state-preserving normal embedding $\iota: (C,\psi) = (C_1, \psi_1) * \dots * (C_n,\psi_n) \to (M_n,\omega_n)$.

To show that the embedding is with expectation, it suffices to show that it commutes with the action of the modular group, i.e.\ $\sigma_t^{\omega_n} \circ \iota = \iota \circ \sigma_t^{\psi}$.  To see this, note that the embedding of $(C_j,\psi_j)$ into $(B_j,\varphi_j) \ootimes (C_j,\psi_j)$ commutes with the action of the modular group.  Then because the embeddings associated to the amalgamated free products also respect the modular group, we see that the embedding of $(C_j,\psi_j)$ into $(M_j,\omega_j)$ and hence into $(M_n,\omega_n)$ respects the modular group.  Meanwhile, the embedding of $(C_j,\psi_j)$ into the free product $(C,\psi)$ also respects the modular group.  Since $(C,\psi)$ is generated by the $C_j$'s, it follows that the embedding of $(C,\psi)$ into $(M_n,\omega_n)$ respects the modular group.
\end{proof}

\begin{lem} \label{lem: word decomposition 3}
Regard $C$ as a von Neumann subalgebra of $(M_n,\omega_n)$.  Let $H_0$ be the span of $K_w$ for alternating words $w$ on $\{0,\dots,n\}$ which start with the letter $0$.  Then for $\xi$, $\eta \in H_0$ and $x, y \in C$, we have
\begin{equation} \label{eq: inner product equality}
\omega_n((x\xi)^*(y\eta)) = \psi(x^*y) \omega_n(\xi^* \eta).
\end{equation}
\end{lem}

\begin{proof}
By linearity, it suffices to consider the following case:  Let $w_0 = j_1 \dots j_\ell$ and $w_0' = j_1' \dots j_{\ell'}'$ be words in $\{0,\dots,n\}$ that start with zero, and suppose that $\xi = \xi_1 \dots \xi_\ell$ and $\eta = \eta_1 \dots \eta_{\ell'}$ are products satisfying the conditions (a)-(c) of Lemma \ref{lem: word decomposition} with respect to $w_0$ and $w_0'$, respectively.  Let $w_1 = k_1 \dots k_m$ and $w_1' = k_1' \dots k_{m'}'$ be alternating words on $\{1,\dots,n\}$ and suppose that $x = x_1 \dots x_m$ and $y = y_1 \dots y_{m'}$ are products of elements in $\ker(\psi)$ with $x_i \in C_{k_i}$ and $y_i \in C_{k_i'}$.

Note that the products $x_1 \dots x_m$ and $y_1 \dots y_{m'}$ also satisfy conditions (a)-(c) of Lemma \ref{lem: word decomposition}.  In fact, the products $x_1 \dots x_m \xi_1 \dots \xi_\ell$ and $y_1 \dots y_{m'} \eta_1 \dots \eta_{\ell'}$ satisfy (a)-(c) as well.  Indeed, since $j_1 = j_1' = 0$, the words are alternating.  Moreover, since all the occurrences of $0$ are in the right part of the word, we see that conditions (b) and (c) are satisfied.

The proof of the claim proceeds by induction on $m$ and $m'$.

\textbf{Case 1:} Suppose that $m$ and $m'$ are both zero.  Then $x$ and $y$ are empty products, which by convention are $1$.  Hence, \eqref{eq: inner product equality} is trivial. 

\textbf{Case 2:} Suppose that $m = 0$ and $m' > 0$.  Since $w_0$ and $w_0'$ start with $0$ and $w_1'$ does not contain $0$, the word $w_0^* w_1' w_0' = j_{\ell} \dots j_1 k_1' \dots k_{m'}' j_1' \dots j_{\ell'}'$ is alternating.  Moreover, the product $\xi_\ell^* \dots \xi_1^* y_1 \dots y_{m'} \eta_1 \dots \eta_{\ell'}$ satisfies (a), (b), (c') of Lemma \ref{lem: word decomposition 2} for this word, using the fact that $\xi_1 \dots \xi_\ell$ and $y_1 \dots y_{m'} \eta_1 \dots \eta_{\ell'}$ satisfy (a)-(c).  Hence, by Lemma \ref{lem: word decomposition 2}, we have
\[
\omega_n(\xi^* y \eta) = \omega_n(\xi_\ell^* \dots \xi_1^* y_1 \dots y_{m'} \eta_1 \dots \eta_{\ell'}) = 0.
\]
On the other hand, by free independence of the $C_j$'s, $\psi(y) = 0$.  Thus, both sides of \eqref{eq: inner product equality} are zero.

\textbf{Case 3:} Suppose that $m > 0$ and $m' = 0$.  The argument is symmetrical to Case 2.

\textbf{Case 4:} Suppose that $m > 0$ and $m' > 0$ and $k_1 \neq k_1'$.  Then the word $w_0^* w_1^* w_1' w_0'$ is alternating and the product $\xi_\ell^* \dots \xi_1 x_m^* \dots x_1^* y_1 \dots y_{m'} \eta_1 \dots \eta_{\ell'}$ satisfies conditions (a), (b), (c') of Lemma \ref{lem: word decomposition 2}. Therefore, $\omega_n(\xi^*x^*y\eta) = 0$.  On the other hand, since $w_1^*w_1'$ is an alternating word on $\{1,\dots,n\}$, we have $\psi(x^*y) = 0$ using free independence of $C_1$, \dots, $C_n$.  Hence, both sides of \eqref{eq: inner product equality} are zero.

\textbf{Case 5:} Suppose that $m > 0$ and $m' > 0$ and $k_1 = k_1'$.  Then the word
\[
j_\ell \dots j_1 k_m \dots k_1 k_2' \dots k_{m'}' j_1' \dots j_{\ell'}'
\]
is alternating and the product
\[
\xi_\ell^* \dots \xi_1^* x_m^* \dots x_2^* (x_1^*y_1 - \psi(x_1^*y_1)) y_2 \dots y_{m'} \eta_1 \dots \eta_{\ell'}
\]
satisfies conditions (a), (b), (c') of Lemma \ref{lem: word decomposition 2}.  Therefore,
\[
\omega_n(\xi^* x_m^* \dots x_2^* (x_1^*y_1 - \psi(x_1^*y_1)) y_1 \dots y_{m'} \eta) = 0,
\]
or
\[
\omega_n(\xi^* x^* y \eta) = \psi(x_1^*y_1) \omega_n(\xi^* x_m^* \dots x_2^* y_2 \dots y_{m'} \eta).
\]
Using the induction hypothesis,
\[
\omega_n(\xi^* x^* y \eta) = \psi(x_1^*y_1) \psi(x_m^* \dots x_2^* y_2 \dots y_{m'}) \omega_n(\xi^*\eta).
\]
By applying the same equality in the case that $\xi = \eta = 1$ (or using free independence),
\[
\psi(x^*y) = \psi(x_1^*y_1) \psi(x_m^* \dots x_2^* y_2 \dots y_{m'}),
\]
and thus, we obtain \eqref{eq: inner product equality} as desired.
\end{proof}

\begin{proof}[Proof of Theorem \ref{thm: coarseness 1}]
In Lemma \ref{lem: freeness of C}, we established the freeness of $C_1$, \dots, $C_n$ and the embedding of $(C,\psi)$ into $(M_n,\omega_n)$ with expectation.

Let $H_0$ be as in Lemma \ref{lem: word decomposition 3} and let $H$ be the closure of $H_0$ in $L^2(M_n,\omega_n)$.  We claim that there is an isometric isomorphism
\begin{equation} \label{eq: constructing isometry}
	\Phi: L^2(C,\psi) \otimes_{\C} H \to L^2(M_n,\omega_n)
\end{equation}
given by $x \otimes \xi \mapsto x \xi$ for $x \in C$ and $\xi \in H$.  Lemma \ref{lem: word decomposition 3} shows that
\[
\ip{x \otimes \xi, y \otimes \eta}_{L^2(C,\psi) \otimes_{\C} H} = \ip{x\xi, y \eta}_{L^2(M_n,\omega_n)}
\]
when $\xi$, $\eta \in H_0$ and $x, y \in C$.  We then extend the isometry to all of $L^2(C,\psi) \otimes_{\C} H$ by continuity. That this map is surjective follows from Lemma \ref{lem: word decomposition}. Indeed, consider any alternating word $w = j_1\dots j_\ell$ on $\{0,\dots,n\}$ and product $x_1\dots x_\ell$ satisfying conditions (a)-(c) of Lemma \ref{lem: word decomposition} with respect to $w$. If the letter 0 appears in $w$, let $j_i$ be its first appearance. Then $x_1\dots x_{i-1} \in C$ and $x_i\dots x_\ell \in H_0$, so $x_1\dots x_\ell$ is in the range of the map above. If the letter 0 does not appear in $w$, then $x_1\dots x_\ell \in C$, $1 \in H_0$, and $x_1\dots x_\ell = (x_1\dots x_\ell) \cdot 1$ is in the range as well.

Next, we claim that $H$ is a right $A$-submodule of $L^2(M_n,\omega_n)$.  First, note that the inclusion of $A$ into $M_n$ admits state-preserving conditional expectation, since $(M_n,\omega_n)$ was formed from $(A,\varphi)$ by taking amalgamated free products.  Equivalently, the inclusion $A \to M_n$ respects the modular groups.  We claim that if $\xi \in H$ and $a \in A$, then $Ja^*J \xi \in H$.  It suffices to prove the claim when $\xi \in H_0$.  Moreover, by modular theory, analytic elements for the modular group are dense in $A$, so it suffices to prove the claim when $a$ is analytic.  This implies that $J a^* J$ is implemented as right multiplication by some $\tilde{a} \in A$, and so $Ja^*J \xi = \xi \tilde{a}$.  By linearity, suppose that $\xi$ is a product $\xi_1 \dots \xi_\ell$ associated to a word $w = j_1 \dots j_\ell$ satisfying the conditions of Lemma \ref{lem: word decomposition} with $j_1 = 0$.  If $j_\ell = 0$ as well, then $\xi_1 \dots \xi_{\ell-1} (\xi_\ell a)$ is also in $K_w$.  If $j_\ell \neq 0$, then $\xi_1 \dots \xi_\ell a$ is in $K_{w0}$ (and $w0$ also starts with $0$).  We thus have $\xi \tilde{a} \in H_0$ as desired.

Next, we show that $\Phi$ is a $C$-$A$-bimodule map.  Given $\zeta \in L^2(C,\psi) \otimes_{\C} H$ and $c \in C$ and $a \in A$, we claim that $\Phi((c \otimes Ja^*J) \zeta) = c Ja^*J \Phi(\zeta)$.  By linearity and density, it suffices to consider the case where $\zeta = x \otimes \xi$ where $x \in C$ and $\xi \in H_0$.  We can also assume that $a$ is analytic for the modular group, so $Ja^*J$ is implemented by right multiplication by some $\tilde{a} \in A$. Then
\begin{align*}
	\Phi((c \otimes Ja^*J)(x \otimes \xi)) &= \Phi(cx \otimes \xi \tilde{a}) \\
	&= cx \xi \tilde{a} \\
	&= c \Phi(x \otimes \xi) \tilde{a} \\
	&= c Ja^*J \Phi(x \otimes \xi).
\end{align*}
This completes the proof of the bimodule isomorphism.
\end{proof}

\begin{thm} \label{thm: coarseness 2}
	Consider the setup of Notation \ref{not: free product setup}.  Fix another von Neumann algebra $(D,\rho)$ and let
	\[
	(M,\omega) = (M_n, \omega_n) *_{(C,\psi)} [(C,\psi) \ootimes (D,\rho)].
	\]
	Then $L^2(M,\omega) \ominus L^2(M_n,\omega_n)$ is a coarse bimodule over $A$.
\end{thm}

\begin{proof}
	Using the bimodule decomposition for amalgamated free products (Lemma \ref{lem: afp-bimod}),
	\begin{multline*}
		L^2(M,\omega) \ominus L^2(M_n,\omega_n) \cong \bigoplus_{k \geq 0} L^2(M_n, \omega_n) \otimes_C \cH_1 \otimes_C (\cH_2 \otimes_C \cH_1)^{\otimes_C k} \otimes_C L^2(M_n,\omega_n)
	\end{multline*}
	as a bimodule over $M_n$, where
    \begin{align*}
        \cH_1 &= L^2[(C,\psi) \ootimes (D,\rho)] \ominus L^2(C,\psi)\\
        \cH_2 &= L^2(M_n,\omega_n) \ominus L^2(C,\psi).
    \end{align*}
    Hence, this also holds as a bimodule over $A$. Since $L^2(M_n,\omega_n)$ is a coarse $C$-$A$-bimodule by Theorem \ref{thm: coarseness 1} and this $C$-$A$-bimodule appears at the right end, we obtain that $L^2(M,\omega) \ominus L^2(M_n,\omega_n)$ is a coarse bimodule over $A$ as desired.
\end{proof}

\section{Exotic full $\mathrm{II}_1$ factors with no property (T) phenomena} \label{sec: exotic factor construction}

Recall the definition of the 3-handle construction, introduced in \cite{gao20253handleconstructionii1factors}. For a $\mathrm{II}_1$ factor $N$ and $u,v\in \mathcal{U}(N)$, define $\Omega_1(N,u)= N*_{W^*(u)}\left(W^*(u)\ootimes L(\mathbb{Z})\right)$  and denote the Haar unitary  generating the $L(\mathbb{Z})$ on the right as $\omega(u)$. Define  $\Omega_2(N,u,v)= N*_{W^*(u,v)}(W^*(u,v)\ootimes 
 L(\mathbb{Z}))$. 
Let $\sigma=(\sigma_1,\sigma_2):\mathbb N\rightarrow\mathbb N\times\mathbb N$ be a bijection such that $\sigma_1(n)\leq n$, for every $n\in\mathbb N$. Assume that $M_1,\ldots,M_n$ have been constructed, for some $n\in\mathbb N$.
Let $\{(u_1^{n,k},u_2^{n,k})\}_{k\in\mathbb N}\subset \mathcal{U}(M_n)$ be a $\|\cdot\|_2$-dense sequence.
We define $$M_{n+1}:=\Omega_2(\Omega_1(\Omega_1(M_n,u_1^{\sigma(n)}),u_2^{\sigma(n)}), \omega(u_1^{\sigma(n)}), \omega(u_2^{\sigma(n)})).$$ Denote by $\Theta(M_1)$ the inductive limit of the $M_i$'s constructed above. Our methods allow us to prove fullness of such factors without any need to appeal to property $(T)$ phenomena. 

\begin{thm}\label{fullness of 3 handle}
    If $M_1$ is an arbitrary $\mathrm{II}_1$ factor, $\Theta(M_1)$ is a full $\mathrm{II}_1$ factor. 
\end{thm}

\begin{proof}
    Since $W^*(u_i^{\sigma(n)})$ is always amenable, by Proposition \ref{prop: amen-afp-wk-coarse}, we get that for all $n,$ the inclusions $M_n\subset \Omega_1(M_n,u_1^{\sigma(n)})$ and $\Omega_1(M_n,u_1^{\sigma(n)}) \subset \Omega_1(\Omega_1(M_n,u_1^{\sigma(n)}),u_2^{\sigma(n)})$ are weakly coarse.
    By Theorem \ref{thm: coarseness 2}, $\Omega_1(\Omega_1(M_n,u_1^{\sigma(n)}),u_2^{\sigma(n)}) \subset M_{n+1}$ is weakly coarse for all $n$ (where $A = M_n$, $C = W^*(\omega(u_1^{\sigma(n)}),\omega(u_2^{\sigma(n)})),$ and $D \cong L\Z$). By Proposition \ref{prop: lim-of-coarse}, we get that $M_k\subset \Theta(M_1)$ is weakly coarse as well, for all $k$. It now suffices to verify that $M_2$ is full in order to apply Proposition \ref{prop: full-and-coarse} to conclude that $\Theta(M_1)$ is full. We note that if $u\in M_1$, then $E_C(u) = \tau(u)$ where $C = W^*(\omega(u_1^{\sigma(1)}),\omega(u_2^{\sigma(1)}))\cong L\F_2$ per Theorem \ref{thm: coarseness 1}. Taking $D \cong L\Z$ in Theorem \ref{thm: coarseness 2}, we can find $v,w$ unitaries in $D$ such that $E_C(v)=E_C(w) = E_C(w^*v) = 0.$ By \cite[Theorem~6.1]{Ioa15}, we deduce that $M_2'\cap M_2^\cU \subset M_2' \cap C^\cU \subset C'\cap C^\cU = \C1$ as $L\F_2$ is full.
\end{proof}

\begin{thm}
    If no corner of $M_1$ contains a diffuse property (T) subalgebra, then $\Theta(M_1)$ contains no diffuse property (T) subalgebra. In particular $\Theta(L(\mathbb{F}_2))$ does not contain any diffuse property (T) subalgebra. 
\end{thm}

\begin{proof}
    Suppose towards a contradiction that $N\subset \Theta(M_1)$ has property (T). Since $\Theta(M_1)$ is an inductive limit of $M_n$, by considering that the expectations from $\Theta(M_1)$ onto $M_n$ converge pointwise to the identity and that $N$ has (T), we deduce that these expectations converge uniformly on $N.$ This in turn implies that a corner of $N$ embeds into a corner of $M_n$ for some $n$ \cite[Theorem~4.6]{christensen1979} (see also \cite{PopaCorr} and \cite[Theorem~3.26]{GKEPTconjugacy2025}). Since all corners of $N$ have property (T), all deformations converge uniformly on $N$. By \cite[Theorem~5.4]{PV09} (c.f. \cite[Theorem~1.1]{IPP}), since $M_n$ is an amalgamated free product, corners of $N$ must intertwine into one of the two factors in the amalgamation. In each step where we obtain $M_n$ from $M_{n-1}$, the right side of the amalgamation is either $L\Z \ootimes L\Z$ or $L\F_2\ootimes L\Z$, which both have Haagerup's property and therefore do not admit diffuse property (T) subalgebras. Therefore we deduce that a corner of $N$ intertwines into $M_{n-1}$, implying that a corner of $N$ embeds into a corner of $M_{n-1}$. Continuing recursively, we deduce that a corner of $N$ embeds into a corner of $M_1$, a contradiction.
\end{proof}

\section{A remark on constructing exotic full $\mathrm{III}_1$ factors} \label{sec: type III setting}

We are able to prove the following result using our machinery:

\begin{prop}
    There exists a full $\mathrm{III}_1$ factor $M$ that satisfies the following property: for every pair of unitaries $u,v\in \mathcal{U}(M)$ such that $\{u\}'',\{v\}''$ admit state preserving expectations, then there exists Haar unitaries $w_1,w_2,w_3\in M^\mathcal{U}$, such that $[u,w_1]=[w_1,w_2]=[w_2,w_3]= [w_3,v]=0$. 
\end{prop}

\begin{proof}
    Let $M_0$ be any full I$\mathrm{II}_1$ factor and $M_1=M_0*L(\mathbb{Z})$. Then we have $M_1$ is a full I$\mathrm{II}_1$ factor admitting a masa $A$ (the free copy of $L(\mathbb{Z})$) inside the centralizer of the free product state. Now we enumerate countable $\|.\|_2$-dense subsets of pairs of finite dimensional unitaries $\{(u_1^{n,k},u_2^{n,k})\}_{k\in\mathbb N}\subset \mathcal{U}(M_n)$ under expectation in $M_n$. Then following the 3-handle construction amalgamating over the corresponding finite dimensional abelian subgalgebras under expectation, we obtain $\Theta(M_1)$. Note that by construction, we have that $A$ is also a masa in $\Theta(M_1)$. This follows from applying \cite[Proposition 3.3]{Ueda_2013} and using the fact that $A$ cannot intertwine into the amalgam in the each of the three stages of each step. Indeed, in the first two stages the amalgams are finite dimensional (recall $A$ is diffuse).  In the third stage $L^2(\Omega_1(\Omega_1(M_n,u_1^{\sigma(n)}),u_2^{\sigma(n)}))$ is a coarse bimodule over $M_n$ and $L(\mathbb{F}_2)$ by Theorem \ref{thm: coarseness 2}, and hence a diffuse subalgebra of $M_n$ cannot intertwine into $L(\mathbb{F}_2)$, which is the amalgam for the third step, so we can again apply \cite[Proposition 3.3]{Ueda_2013}.  Now applying \cite[Proposition 2.7]{Ando_2019}, we see that this is a type $\mathrm{III}_1$ factor since $M_1\subset \Theta(M_1)$ is with expectation and additionally $A$ is a masa under expectation in $M_1$ and is also a masa in $\Theta(M_1)$. This factor is additionally full by Theorem \ref{thm: coarseness 2} and Proposition \ref{prop: lim-of-coarse}. 
\end{proof}

One may conjecture that the above factor $M$ is not elementarily equivalent to the free Araki Woods factor, via results from \cite{hayes2024generalsolidityphenomenaanticoarse}, which show lack of sequential commutation. Unfortunately, this cannot be immediately deduced because of the following subtle reason. Our construction above does not connect arbitrary pairs of unitaries, but rather only unitaries with state-preserving expectation.  One could attempt to perform the construction merely assuming \emph{some} faithful normal expectation onto the algebra generated by each unitary (which would always exist if this is finite-dimensional), and in fact some of the computations in \S \ref{sec: bimodule computation} would go through, but the problem is that Lemma \ref{lem: freeness of C} crucially uses that the expectations from $M_n$ onto the $C_j$'s are consistent with the same state, which would not be the case if the expectations from $A$ onto $B_j$ are not consistent with the same state.  Without Lemma \ref{lem: freeness of C}, we would not have a conditional expectation onto $C$ to finish the construction in Theorem \ref{thm: coarseness 2}. If it was the case that for any $u\in \mathcal{U}(M)$ and every $\epsilon>0$, there exists a unitary $v\in M$, such that $\|u-v\|_2<\epsilon$ and $v$ is contained in an amenable subalgebra of $M$ with state preserving expectation, then we can conclude the result by using these amenable subalgebras as the $B_j$'s in \S \ref{sec: bimodule computation}. However this seems far from possible in general for full $\mathrm{III}_1$ factors.

\bibliographystyle{amsalpha}
\bibliography{inneramen}

\end{document}